\documentclass[10pt]{amsart}

\usepackage{amsfonts}
\usepackage{amssymb}

\theoremstyle{plain}
\newtheorem{thm}{Theorem}[section]
\newtheorem{prop}[thm]{Proposition}
\newtheorem{cor}[thm]{Corollary}
\newtheorem{lemma}[thm]{Lemma}
\newtheorem{ques}[thm]{Question}

\newcommand{\cal}{\mathcal}
\renewcommand{\mod}[1]{\left\lvert #1 \right\rvert}
\newcommand{\inp}[2]{\left\langle #1,#2 \right\rangle}
\newcommand{\norm}[1]{\left\| #1 \right\|}

\newcommand{\bb}[1]{\mathbb{#1}}

\begin{document}

\title{Representations of Logmodular Algebras}
\author[V. I. Paulsen]{Vern I. Paulsen}
\address{Department of Mathematics, University of Houston, Houston, TX, 77204}
\email{vern@math.uh.edu}

\author[M. Raghupathi]{Mrinal Raghupathi}
\address{Department of Mathematics, University of Houston, Houston, TX, 77204}
\email{mrinal@math.uh.edu}
\date{\today}
\urladdr{http://www.math.uh.edu/~mrinal}
\thanks{This research was supported in part by NSF grant
  DMS-0600191.} 

\begin{abstract}
  We study the question of whether or not contractive representations
  of logmodular algebras are completely contractive. We prove that a
  2-contractive representation of a logmodular algebra extends to a
  positive map on the enveloping C*-algebra, which we show generalizes
  a result of Foias and Suciu on uniform logmodular algebras. Our
  proof uses non-commutative operator space generalizations of
  classical results on 2-summing maps and semispectral measures. We
  establish some matrix factorization results for uniform logmodular
  algebras.
\end{abstract}

\maketitle

\section{Introduction}

Let $\cal B$ be a unital $C^\ast$-algebra and let $\cal A \subseteq
\cal B$ be a unital subalgebra. By a {\em representation} of $\cal A$
on a Hilbert space $\cal H$ we mean a unital, contractive
homomorphism, $\rho: \cal A \to B(\cal H)$, where $B(\cal H)$ denotes
the $C^\ast$-algebra of bounded operators on $\cal H.$ A
representation of $\cal A$ is said to have a \textit{$\cal
  B$-dilation} if there exists a Hilbert space $\cal K$, an isometry,
$V: \cal H \to \cal K$ and a unital $\ast$-homomorphism, $\pi: \cal B
\to B(\cal K),$ such that $\rho(a)= V^\ast\pi(a)V,$ for every $a \in
\cal A.$ By a famous result of Arveson~\cite{Ar}, $\rho$ has a $\cal
B$-dilation if and only if $\rho$ is a {\em completely contractive}
map. For this reason there is a great deal of interest in results that
imply that contractive representations are automatically completely
contractive.

In particular, the Sz.-Nagy dilation theorem implies that contractive
representations of the disk algebra are completely contractive and it
is proven in~\cite{MM},~\cite{PP} and~\cite{DPP} that contractive
representations of many finite-dimensional CSL algebras, including
algebras of block upper triangular matrices, are completely contractive.

However, it is still unknown if contractive representations of
$H^{\infty}(\bb D)$ or of the upper triangular operators on
$\ell^2(\bb N)$ are automatically completely contractive. It is known
that weak*-continuous contractive representations of both these
algebras are completely contractive. The case of $H^{\infty}(\bb D)$
follows readily from Sz.-Nagy's theorem and the case of nest algebras
is done in~\cite{PPW}.

Both these algebras are examples of {\em logmodular} algebras and it
is still unknown if, more generally, contractive representations of
logmodular algebras are always completely contractive.

If $\rho:\cal{A}\to B(\cal{H})$ is a representation, then we say that
$\rho$ is $R_n$-contractive if $\rho_{1,n}:M_{1,n}(\cal{A})\to
B(\cal{H}^{(n)},\cal{H})$, defined by $\rho_{1,n}((a_1,\ldots,a_n))=
(\rho(a_1),\ldots,\rho(a_n))$, is contractive. The notion of
$C_n$-contractive is defined by using columns instead of rows. 

Foias and Suciu~\cite{FS} proved, in modern terminology, that every
$R_2$- and $C_2$-contractive representation of a logmodular uniform
algebra is automatically completely contractive. Thus, by their result
a $R_2$- and $C_2$-contractive representation of $H^{\infty}(\bb D)$
is completely contractive. However, it is still not known if every
$R_2$- and $C_2$-contractive representation of a general logmodular
algebra is automatically completely contractive. In particular, it is
not known if every $R_2$- and $C_2$-contractive representation of the
upper triangular operators on $\ell^2(\bb N)$ is completely
contractive.

In these notes we try to extend the result of Foias and Suciu to
general logmodular algebras. We prove that if $\cal A$ is a logmodular
subalgebra of a $C^\ast$-algebra $\cal B,$ then every representation
of $\cal A$ on $\cal H$ that is both $R_2$- and
$C_2$-contractive extends to a positive map from $\cal B$ to $B(\cal
H)$ and that this positive map is unique.  The Foias-Suciu result then
follows as a corollary.

A result of Pietsch, Pelczy\'{n}ski and Lindenstrauss~\cite{Pi, LP},
on the existence of ``dominating'' measures for 2-summing maps plays a
central role in Foias-Suciu's proof. We give a new proof of this
result using operator space methods and obtain a non-commutative
analogue that we need for our generalization.

The second key element in the Foias-Suciu proof is Mlak's
theory~\cite{Ml} of semispectral measures and we obtain an analogue of
this theory for non-commutative $C^\ast$-algebras.

Recent work of Blecher and Labuschange~\cite{BL,BL2} has shown that
many classical function theoretic results on logmodular algebras
extend to the ``non-commutative'' logmodular algebras. Their results
consider only completely contractive representations and do not
address the above problems, but we are able to use many of their
ideas. For example, they prove that every completely contractive
representation of an arbitrary logmodular algebra extends uniquely to
a two-positive map on the $C^\ast$-algebra. In their case, the
existence of a completely positive extension is automatic by Arveson's
extension theorem, while the issue in their case is the uniqueness. In
our case, the difficulty is the existence of the extension at all,
while the techniques to prove uniqueness are essentially contained in
their paper~\cite{BL}.

One limitation of the theory of logmodular algebras is the lack of
many natural examples. We prove that the only logmodular CSL algebras
of matrices are the block upper triangular matrices. Thus, by a result
of~\cite{PP} every contractive representation of these types of
logmodular algebras is completely contractive. Possibly every
logmodular subalgebra of the algebra of matrices is of this form. This
result also indicates that possibly every logmodular completely
distributive CSL algebra is a nest algebra.

\section{Definitions and Examples}

Let $X$ be a compact Hausdorff space. Let $C(X)$ denote the continuous
complex valued functions on $X$ and let $C(X)_+$ denote the set of
positive elements in $C(X)$. If $S\subset C(X),$ then let
$\cal{S}^{-1}$ denote the set of elements in $\cal{S}$ that are
invertible in $\cal{S}$. A {\em logmodular algebra} on $X$ is a
subalgebra $\cal{A}$ of $C(X)$ such that the set $\{\log \mod{a}\,:\,
a\in \cal{A}^{-1}\}$ is dense in the real-valued functions on $C(X)$,
which we denote $C_\bb{R}(X)$. The following well-known fact is
immediate.

\begin{lemma}
  $\cal{A} \subseteq C(X)$ is logmodular if and only if the set
  $\{\mod{f}^2\,:\, f\in \cal{A}^{-1}\}$ is dense in the set of
  positive invertible functions on $X$.
\end{lemma}

This equivalent characterization of logmodularity motivates the
following definition. Let $\cal{B}$ be a $C^\ast$-algebra and let
$\cal{A}$ be a subalgebra. Following ~\cite{BL}, we say that $\cal{A}$
is {\em logmodular} in $\cal{B}$ if the set $\{a^\ast a\,:\, a\in
\cal{A}^{-1}\}$ is dense in $\cal{B}_+^{-1}$.

From the invertiblity of the $a\in \cal{A}^{-1}$, it follows that the
set $\{aa^\ast\,:\, a\in \cal{A}^{-1}\}$ is also dense in
$\cal{B}_+^{-1}$.  Thus, if $\cal{A}$ is logmodular, then so is the
subalgebra $\cal{A}^\ast$.

Some algebras possess a property called {\em stong logmodularity} or
{\em factorization}. This means that every positive invertible element
in $\cal{B}$ is of the form $a^\ast a$ for some $a\in \cal{A}^{-1}$.

Examples of logmodular algebras include the disk algebra $A(\bb{D})
\subseteq C(\bb T)$, a conseqeunce of the Fej\'{e}r-Riesz theorem on
the factorization of positive trigonometric polynomials. However, the
disk algebra is not a logmodular subalgebra of $C(\bb D^-).$ More
generally, $H^\infty(\bb{D})$ is (strongly) logmodular when viewed as
a subalgebra of $L^\infty(\bb T)$.

An algebra $\cal{A}\subseteq C(X)$ is called \textit{Dirichlet} if and
only if $\Re(\cal{A})$ is uniformly dense in $C_\bb{R}(X)$. This is
equivalent to $\cal{A}+\overline{\cal{A}}$ being uniformly dense in
$C(X)$. Every Dirichlet algebra $\cal{A}$ is logmodular on
$C(X)$. However, the algebra $H^\infty$ is not Dirichlet. We refer the
reader to \cite[Section~3]{KH1} for more about these examples.

If $\bb{A}:=\{z\in\bb{C}\,:\, 0<r \leq |z|\leq 1\}$ is a closed
annulus, then the algebra of functions that are continuous on
$\partial \bb{A}$ and analytic on the interior of $\bb{A}$ provides an
example of an algebra that is not logmodular. However, a deep result
of Agler shows that every representation of this algebra is completely
contractive~\cite{agler}.

We now give some examples of non-commutative logmodular algebras. The
Cholesky decomposition shows that the algebra of upper triangular
matrices $\cal{T}_n$ is logmodular in $M_n$ and more generally, nest
subalgebras for countable nests are logmodular subalgebras of $B(\cal
H)$ since they have factorization by ~\cite{Ar}.

We now prove that among a certain family of subalgebras of matrix
algebras, the upper triangular matrices are essentially the only
logmodular algebras.  We let $M_n$ denote the $n \times n$ matrices
and let $\cal D_n \subseteq M_n$ denote the subalgebra of diagonal
matrices. Every subalgebra $\cal A, \cal D_n \subseteq \cal A
\subseteq M_n$ is known to be the algebra of matrices left invariant
by a {\em commutative lattice of subspaces.} Moreover, every
subalgebra of $M_n$ that is left invariant by a commuting lattice of
subspaces can be seen to be unitarily equivalent to such an
algebra. Such algebras are called {\em CSL algebras.} For further
details on CSL algebras see ~\cite{Da}. While the results below do not
need this characterization, the characterization serves to put the
result in perspective.

We call $\cal A\subseteq M_n$ an algebra of \textit{block upper triangular
matrices} if there are positive integers $n_1,\ldots,n_k$ such that
$n=n_1+\ldots+n_k$, and $\cal{A}$ is the set of matrices of the form
$(A_{i,j})$, where $A_{i,j}\in M_{n_i,n_j}$ and $A_{i,j}=0$ whenever
$i>j$.

\begin{thm} Let $\cal D_n \subseteq \cal A \subseteq M_n.$ Then $\cal
  A$ is a logmodular subalgebra of $M_n$ if and only if after a
  permutation of the basis, $\cal A$ is an algebra of block upper
  triangular matrices.
\end{thm}
\begin{proof} If $\cal A$ is an algebra of block upper triangular
  matrices, then it contains the algebra of upper triangular matrices
  and hence is logmodular.

  Conversely, assume that $\cal A$ is logmodular. Since $\cal D_n
  \subseteq \cal A,$ we easily see that $\cal A$ is the span of the
  matrix units, $\{ E_{i,j} \}$ that belong to $\cal A.$

  By a compactness argument, it follows that every positive, $P \in
  M_n,$ factors as $P=A^*A,$ with $A \in \cal A.$ Let $J$ denote the
  rank one positive matrix with all entries equal to 1 and write
  $J=A^*A$ with $A \in \cal A.$ If we let $R_1,...,R_n$ denote the
  rows of $A$, then $J= R_1^*R_1 + ... + R_n^\ast R_n$ and each
  $R_i^*R_i$ is a rank one positive. From this it follows that for
  each $i,$ either $R_i^*R_i=0$ or $R_i^*R_i$ is a positive multiple
  of $J$.

  Hence, we have a row $k$ such that $R_k^*R_k$ is a positive multiple
  of $J$ and so each entry of this row must be non-zero. Thus, $\{
  E_{k,j}: j=1,...,n \} \subseteq \cal A.$ Re-numbering, we have that
  $\{ E_{1,j}: j=1,...,n \} \subseteq \cal A.$ Now let $S= \{ i :
  E_{i,1} \in \cal A \}.$ For $i \in S$ and any $j$, we have that
  $E_{i,j} = E_{i,1}E_{1,j} \in \cal A.$ If $i \notin S, j \in S,$
  then $E_{i,j} \notin \cal A,$ or else, $E_{i,1} = E_{i,j}E_{j,1} \in
  \cal A.$

  Thus, re-numbering so that $S= \{1,...,m \},$ we have that $E_{i,j}
  \in \cal A,$ for all $1 \le i,j \le m,$ while $E_{i,j} \notin \cal
  A,$ whenever, $m < i$ and $1 \le j \le m,$ that is, $\cal A$
  contains the entire upper left hand block.  Relative to this
  decomposition, setting $p=n-m,$ we have that $M_n = \begin{bmatrix}
    M_m & M_{m,p} \\ M_{p,m} & M_p \end{bmatrix}$ and
  \[\cal A = \left\{ \begin{bmatrix} A & B\\ 0 & C \end{bmatrix} : 
    A \in M_m, B \in M_{m,p}, C \in \cal A_1 \right\},\] where $\cal A_1
  \subseteq M_p$ is a subalgebra that contains the diagonal
  subalgebra, $\cal D_p.$

Now if we let $P= \begin{bmatrix} I_m &0 \\0 & Q \end{bmatrix}$ where $Q \in M_p$ is positive, then factoring, $P= A^*A,$ with $A= \begin{bmatrix} X & Y\\ 0 & Z \end{bmatrix} \in \cal A,$ we find that $I_m= X^*X, 0 = X^*Y$ and $Q= Y^*Y+Z^*Z.$ But the first equation implies that $X$ must be a unitary, and hence, $Y=0,$ so that $Q= Z^*Z,$ with $Z \in \cal A_1.$

Thus, $\cal A_1$ is a logmodular subalgebra of $M_p$ and we are done by induction.

\end{proof}

\begin{cor} Let $\cal D_n \subseteq \cal A \subseteq M_n.$ If $\cal A$ is a logmodular subalgebra of $M_n,$ then every contractive representation of $\cal A$ is completely contractive.
\end{cor}
\begin{proof}  In ~\cite{PP} it is proven that every contractive representation of an algebra of block upper triangular matrices is completely contractive.
\end{proof}

We do not know if every logmodular subalgebra of $M_n$ is of this form, but we believe that this is the case.
This result makes it natural to conjecture that every logmodular completely distributive CSL algebra is a nest algebra.

In ~\cite[Proposition~4.3]{BL} it is shown that if $\cal A$ is
logmodular on $ \cal
B,$ then $C^*_e(\cal A)= \cal B,$ where $C^*_e(\cal A)$ denotes Hamana's
boundary algebra, thus $\cal B$ is already the smallest $C^\ast$-algebra that
any completely isometric representation of $\cal A$ can generate.




\section{Noncommutative Analogues of Some Classic Results}

In this section we prove analogues of a classic factorization result for 2-summing maps and obtain a noncommutative analogue of Mlak's~\cite{Ml} theory of semispectral measures.

We first turn our attention to $2$-summing maps and their relation to completely bounded maps. Given a Hilbert space $\cal H,$ we form two operator spaces, {\em row Hilbert space,} $\cal H_r = B(\cal H, \bb C)$ and {\em column Hilbert space,} $\cal H_c = B(\bb C, \cal H).$ For more details on the structure of these operator spaces, see~\cite{Pa}.

Let $X$ be a compact Hausdorff space and let $\cal{V}$ be a subspace of $C(X)$. A map $\psi:\cal V \to \cal{H}$ is called 2-summing if there exists a constant $c$ such that 
\[\sum_{j=1}^n \norm{\psi(f_j)}^2_{\cal H} \leq c^2 \norm{\sum_{j=1}^n \mod{f_j}^2}_{\infty}\]
for all $n\geq 1$ and $f_1,\ldots,f_n\in\cal{A}$. The least such constant $c$ is denoted $a_2(\psi)$. 

The following result follows from \cite[Proposition~5.11]{Pis}(see
also \cite[Theorem~5.7]{ER}) and the
fact that, $a_2(\psi) = \pi_{2,c}(\psi) = \pi_{2,r}(\psi),$ but we
provide a direct argument.

\begin{lemma}
Let $\cal V \subseteq C(X)$ be a subspace, let $\cal H$ be a Hilbert
space and let $\psi:\cal V \to \cal{H}$ be a linear map. Then the
following are equivalent:
\begin{enumerate}
\item $\psi : \cal V \to \cal H$ is 2-summing,
\item $\psi: \cal V \to \cal H_r$ is completely bounded,
\item $\psi: \cal V \to \cal H_c$ is completely bounded.
\end{enumerate}
Moreover, in this case the 2-summing norm, $a_2(\psi)$ and the two
completely bounded norms are equal.
\end{lemma}

\begin{proof} We only prove the equivalence of (1) and (2) and the
  equality of the 2-summing norm with the completely bounded norm into
  $\cal H_r.$ The proof of the equivalence of (1) and (3) and the other
  equality is identical.

To this end,
let $(h_{i,j})\in M_n(\cal{H}_r)$ and note that 
\begin{align*}
\norm{(h_{i,j})}^2&=\norm{(h_{i,j})(h_{i,j})^\ast}=\norm{(\sum_{k=1}^n\inp{h_{i,k}}{h_{j,k}})}\\
&=\sup \left\{\mod{\sum_{i,j,k=1}^n \inp{\lambda_i h_{i,k}}{\mu_j h_{j,k}}}\,:\, \sum_{i=1}^n \mod{\lambda_i}^2=\sum_{j=1}^n \mod{\mu_j}^2=1\right\}
\end{align*} 
Now, 
\begin{multline}\label{mat1}
\norm{(\psi(f_{i,j}))}^2=\\\sup \left\{\mod{\sum_{i,j,k=1}^n \inp{\lambda_i \psi(f_{i,k})}{\mu_j \psi(f_{j,k})}}\,:\, \sum_{i=1}^n \mod{\lambda_i}^2=\sum_{j=1}^n \mod{\mu_j}^2=1\right\}
\end{multline}
Let $g_k=\sum_{i=1}^n \lambda_i f_{i,k}$ and $h_k=\sum_{j=1}^n \mu_jf_{j,k}$. We have that the quantity in \eqref{mat1} is smaller than 
\begin{align*}
\sum_{k=1}^n \mod{\inp{\psi(g_k)}{\psi(h_k)}}&\leq (\sum_{k=1}^n \norm{\psi(g_k)}^2)^{1/2}(\sum_{k=1}^n \norm{\psi(h_k)}^2)^{1/2}\\
& \leq  a_2(\psi)^2\norm{\sum_{i=1}^n \mod{g_k}^2}^{1/2}\norm{\sum_{j=1}^n \mod{h_k}^2}^{1/2}
\end{align*}

Now, 
$(g_1,\ldots, g_n)=(\lambda_1,\ldots,\lambda_n) (f_{i,j})$ implies 
 \begin{align*}
\sum_{k=1}^n \mod{g_k}^2 &= (g_1,\ldots,g_n)(g_1,\ldots,g_n)^\ast \\
&=(\lambda_1,\ldots,\lambda_n)(f_{i,j})(f_{i,j})^{\ast} (\lambda_1,\ldots,\lambda_n)^\ast\\
& \leq (\lambda_1,\ldots,\lambda_n) \norm{(f_{i,j})}^2 (\lambda_1,\ldots,\lambda_n)^\ast  \leq \norm{(f_{i,j})}^2.
\end{align*}
and similarly we can show $\sum_{k=1}^n \mod{h_k}^2\leq \norm{(f_{i,j})}^2$.

Combining the above inequalities, we get,
$$\norm{(\psi(f_{i,j}))}^2\leq a_2(\psi)^2 \norm{(f_{i,j})}^2.$$
Hence, $\norm{\psi}_{cb} \le a_2(\psi),$ and so every 2-summing map is
completely bounded into the row Hilbert space.

Conversely, if $\psi$ is completely bounded into the row Hilbert
space, then
we have,
$$\sum_{i=1}^n \norm{\psi(f_i)}^2=\norm{(\psi(f_1),\ldots,\psi(f_n))}^2\leq \norm{\psi}^2_{cb}.$$
and so $\norm{\psi}_{cb} \ge a_2(\psi)$. Thus, the map $\psi$ is 2-summing
and the norms agree.
\end{proof}

The result of~\cite{Pi, LP} on the existence of dominating measures is essential to the proof of
Foias-Suciu~\cite{FS}. We now present a non-commutative version of this
result from which the result of~\cite{Pi, LP} follows readily. Note
that if $v \in \cal H_r = B(\cal H, \bb C),$ then $v^*v \in B(\cal H)$
is a rank one positive operator. While if $v \in \cal H_c,$ then $vv^*
\in B(\cal H)$ is a rank one positive operator. The existence of the
state in the following result is related to \cite[Proposition~5.1]{Pis}.

\begin{thm}\label{dominatingstate} Let $\cal B$ be a unital $C^\ast$-algebra, let $\cal X
  \subseteq \cal B$ be an operator subspace and let $\cal H$ be a
  Hilbert space. If $\psi: \cal X \to \cal H_r= B(\cal H, \bb C)$ (respectively, $\cal H_c = B(\bb C, \cal H)$)  is completely
  contractive, then there exists a state, $s: \cal B \to \bb C$ and a
  unital completely positive map $\Phi: \cal B \to B(\cal H)$ such
  that for every $x \in \cal X,$ we have that
$\|\psi(x)\|^2 \le s(xx^*)$ and $\psi(x)^*\psi(x) \le \Phi(x^*x)$(respectively, $\|\psi(x)\|^2 \le s(x^*x)$ and $\psi(x)\psi(x)^* \le \Phi(xx^*)$).
\end{thm}

\begin{proof}
Writing an arbitrary operator $T \in B(\bb C \oplus \cal H)$ as a $2 \times 2$ operator
matrix, we have that $T = \begin{bmatrix} \lambda & r \\ c &
  B \end{bmatrix}$ where $\lambda \in \bb C, r \in \cal H_r, c \in
\cal H_c$ and $B \in B(\cal H).$ That is we identify 
\[B(\bb C \oplus \cal H) = \begin{bmatrix} \bb C & \cal H_r \\ \cal H_c & B(\cal H)\end{bmatrix}.\]
Let 
\[\cal{S}_\cal{X}:=\left\{\begin{bmatrix}\lambda 1 & x \\ y^* & \mu 1 \end{bmatrix}\, :\,
  \lambda, \mu\in \bb{C}, x,y \in \cal{X} \right\} \subseteq M_2(\cal
B),\]
where $1$ denotes the unit of $\cal B.$
Since $\psi$ is completely contractive, the map $\Psi:
\cal{S}_{\cal{X}} \to B(\bb{C} \oplus \cal{H})$ given by
$\Psi\left(\begin{bmatrix}\lambda & x \\ y^* & \mu\end{bmatrix}\right) =
\begin{bmatrix}\lambda & \psi(x) \\ \psi(y)^* & \mu I_{\cal H}\end{bmatrix}$ is
completely positive by ~\cite[Lemma~8.1]{Pa}. Thus, by Arveson's
extension theorem, we may extend $\Psi$ to a completely positive map
on $M_2(\cal B),$
which we still denote by $\Psi.$ Thus,
$\Psi:M_2(\cal B)\to B(\bb{C}\oplus \cal{H})$. By
~\cite[Theorem~8.3]{Pa}, there exists two completely positive maps
$\phi_1:\cal B \to \bb{C}$ and $\Phi_2:\cal B \to B(\cal{H})$ and a
completely bounded map $u$ such that 
\[\Psi\left(\begin{bmatrix}b_{1,1} & b_{1,2} \\ b_{2,1} & b_{2,2}\end{bmatrix} \right)=\begin{bmatrix}\phi_1(b_{1,1}) &
  u(b_{1,2}) \\ u(b_{2,1}^*)^\ast & \Phi_2(b_{2,2})\end{bmatrix}.\] 

 For any $x \in \cal{X}$ the matrix $\begin{bmatrix}xx^* & x \\ x^* & 1\end{bmatrix}\geq 0$ 
which implies that $\begin{bmatrix}\phi_1(xx^*) & \psi(x) \\ \psi(x)^\ast & I\end{bmatrix}$ is positive which yields 
$$\norm{\psi(x)}^2\leq \phi_1(xx^*).$$ Since $\phi_1(1)=1,$ we have that $s= \phi_1$ is the desired state.

On the other hand, $\begin{bmatrix}1 & x \\x^* & x^*x\end{bmatrix} \ge 0,$ which implies that
$\begin{bmatrix}1 & \psi(x) \\ \psi(x)^* & \Phi_2(x^*x)\end{bmatrix}$ is positive, from which it follows that $\psi(x)^*\psi(x) \le \Phi_2(x^*x),$ and $\Phi = \Phi_2$ is the desired unital completely positive map.

For the case where $\psi:X \to \cal H_c$ is completely contractive, we define 
\[\cal S := \left\{ \begin{bmatrix}\lambda & y^* \\ x & \mu I\end{bmatrix}\, :\, \lambda, \mu \in \bb C, x,y \in \cal X \right\}\] 
and $\Psi: \cal S \to B(\bb C \oplus \cal H)$ via $\Psi \left(\begin{bmatrix} \lambda & y^* \\ x & \mu I\end{bmatrix} \right) = \begin{bmatrix}\lambda & \psi(y)^* \\ \psi(x) & \mu I\end{bmatrix}$, and argue as above.
\end{proof}

\begin{cor}[Lindenstrauss-Pelczy\'{n}ski,Pietsch] Let $\cal V
  \subseteq C(X)$ and let $\psi: \cal V \to \cal H$ with $\cal H$ a
  Hilbert space. If $\psi$ is $2$-summing, then there exists a
  positive measure $\mu$ on $X$ with $\mu(X)=1$ such that
  $\norm{\psi(f)}^2\leq  a_2(\psi)^2 \int_X \mod{f}^2 \,d\mu$.
\end{cor}

We now turn our attention to obtaining $C^\ast$-algebraic generalizations of Mlak's theory of semispectral measures~\cite{Ml}.
We shall need the following folklore result about quadratic forms
on Hilbert spaces.

\begin{prop}\label{quadratic} Let $\cal H$ be a complex Hilbert space
  and let $\phi: \cal H \to \bb C$ be a bounded quadratic form, that is:
\begin{enumerate}
\item $\phi(\lambda
  h) = |\lambda|^2 \phi(h),$ for every $\lambda \in \bb C, h \in \cal H,$
\item there exists a constant $C>0,$ with $|\phi(h)| \le C \|h\|^2,$
\item $\phi(h+k) + \phi(h-k) = 2 \phi(h) + 2\phi(k).$
\end{enumerate}
Then there exists a bounded operator, $T \in B(\cal H),$ such
that $\phi(h) = \langle Th,h \rangle,$ for every $h \in \cal H$ and
$\|T\| \le 2C$ and such a $T$ is unique.
\end{prop}
\begin{proof} This result is ~\cite[Exercise~4.18]{Pa}, so we only
  sketch the proof. One defines a two variable function, $\psi: \cal H
  \times \cal H \to \bb C$ by polarization and
  then copies the steps of the simplified proof of the Jordan-von~Neumann theorem
  to verify that $\psi$ is a bounded sesquilinear map.
\end{proof}

Note that if we are given a unital $C^\ast$-algebra $\cal B,$ a Hilbert space $\cal H$ and a positive map $\Phi: \cal B \to B(\cal H),$ then for each vector $h \in \cal H,$ we obtain a positive functional $\gamma(h)$ on $\cal B$ by setting $\gamma(h)(b) = \langle \Phi(b)h,h \rangle.$ The following result characterizes such maps.
We let $\cal B^{\dagger}_+$ denote the set of positive linear functionals on $\cal B.$

\begin{thm}\label{semispectral2} Let $\cal B$ be a $C^\ast$-algebra, let $\cal H$ be a Hilbert space and let $\gamma: \cal H \to \cal B^{\dagger}_+$ be a map. Then there exists a positive linear map, $\Phi: \cal B \to B(\cal H)$ such that $\gamma(h)(b) = \langle \Phi(b)h,h \rangle$ for every $b \in \cal B$ and every $h \in \cal H$ if and only if
\begin{enumerate}
\item $\gamma( \lambda h) = |\lambda|^2 \gamma(h)$ for every $\lambda \in \bb C$ and every $h \in \cal H,$
\item there is a constant $C$ such that $\|\gamma(h)\| \le C \|h\|^2,$ for every $h \in \cal H,$
\item $\gamma(h+k) + \gamma(h-k) = 2 \gamma(h) + 2 \gamma(k),$ for every $h,k \in \cal H.$
\end{enumerate}
Moreover, if $\cal B$ is unital, then $\Phi$ is unital if and only if $\gamma(h)$ is a state for every $\|h\|=1.$
\end{thm}

\begin{proof} For each fixed $b \in \cal B,$ we see that $h \to \gamma(h)(b)$ is a bounded quadratic form. Hence, there exists an operator $\Phi(b) \in B(\cal H)$ such that $\gamma(h)(b) = \langle \Phi(b)h,h \rangle .$  The linearity and positivity of the map $b \to \Phi(b)$ now follows from the fact that each $\gamma(h)$ is linear and positive.  

Finally, $\Phi(I)=I$ if and only if $1 = \langle \Phi(I)h,h \rangle = \gamma(h)(I),$ for every unit vector $h$.
\end{proof}

\begin{cor}[Mlak's Dilation Theorem for Semispectral Measures] Let $X$ be a locally compact Hausdorff space, let $\cal B(X)$ be the $\sigma$-algebra of Borel measurable subsets of $X$ and let $\cal H$ be a Hilbert space and assume that for each $h \in \cal H$ we are given a positive, regular Borel measure $\mu_h$ on $X,$ satisfying:
\begin{enumerate}
\item $\mu_{\lambda h} = |\lambda|^2 \mu_h,$ for every $\lambda \in \bb C$ and every $h \in \cal H,$
\item there is a constant $C$ such that $\mu_h(X) \le C \|h\|^2,$ for every $h \in \cal H,$
\item $\mu_{h+k} + \mu_{h-k} = 2 \mu_{h} + 2 \mu_k,$ for every $h,k \in \cal H.$
\end{enumerate}
Then there exists a Hilbert space $\cal K$, a linear map $V: \cal H \to \cal K$ and a regular Borel projection-valued measure, $E: \cal B(X) \to B(\cal K)$ such that for every Borel set, $B \subseteq X,$ and every vector $h \in \cal H,$ we have $\mu_h(B) = \langle E(B)Vh, Vh \rangle.$
\end{cor}
\begin{proof} By Theorem~\ref{semispectral2} there exists a positive map $\Phi: C_0(X) \to B(\cal H),$ with $\langle \Phi(f) h,h \rangle = \int_X f d \mu_h ,$ for every $h \in \cal H.$ By Stinespring's theorem ~\cite{St}, this positive map is completely positive. Therefore the map $\Phi$ dilates, that is, there exists a $\ast$-homomorphism $\pi: C_0(X) \to B( \cal K)$ and a map $V: \cal H \to \cal K$ such that $\Phi(f) = V^*\pi(f)V.$
The result now follows by letting $E$ be the regular Borel projection-valued measure such that $\pi(f) = \int_X f dE.$
\end{proof}

It would be interesting to find a similar characterization of maps $\gamma$ of the form $\gamma(h)(b) = \langle \Phi(b)h,h \rangle$ such that $\Phi: \cal B \to B(\cal H)$ is a {\em completely} positive map that did not reduce to a simple tautology. 

\section{Representations of Logmodular Algebras}

In this section we obtain our principal results on representations.

Given operator spaces $X$ and $Y$ and a linear map $\psi:X \to Y,$ we call $\psi$ {\em $R_n$-contractive} provided that 
\[\|(\psi(x_1), \ldots, \psi(x_n))\| \le \|(x_1, \ldots, x_n)\|,\] 
for every $x_1, \ldots, x_n \in X$ and {\em row contractive} if it is $R_n$-contractive for every $n$. We define {\em $C_n$-contractive} and {\em column contractive}, analogously.

\begin{lemma}\label{colimpliescc}
Let $X$ be an operator space and $\psi:X\to \cal{H}_c$ (respectively, $\cal H_r$) be a linear map. If $\psi$ is column contractive (respectively, row contractive), then $\psi$ is completely contractive.
\end{lemma}

\begin{proof}
Begin by noting that we may identify $M_{m,n}(\cal{H}_c)$ with $B(\bb{C}^n,\cal{H}^{(m)})$. Let $(x_{i,j})\in M_{m,n}(X)$ and let $\lambda=\begin{bmatrix}\lambda_1 \\ \vdots \\\lambda_n\end{bmatrix}\in \bb{C}^n$, with $\sum_{j=1}^n \mod{\lambda_j}^2=1$. Write $(x_{i,j})=[C_1,\ldots,C_n]$, $C_j\in M_{m,1}(X)$.  We have, 
\[\norm{(\psi(x_{i,j}))\lambda}=\norm{\psi_{m,1}\left(\sum_{j=1}^n \lambda_j C_j\right)}\leq \norm{\sum_{j=1}^n \lambda_j C_j}=\norm{(x_{i,j})\lambda}\leq \norm{(x_{i,j})}. \]
Hence, $\psi$ is completely bounded and $\norm{\psi}_{cb}=\norm{\psi}_{col}$. 

The row case is analogous.
\end{proof}

\begin{thm}\label{tworowimpliescc}
Let $\cal B$ be a $C^\ast$-algebra, let $\cal A \subseteq \cal B$ be logmodular and let $\rho:\cal{A}\to B(\cal{H})$ be a $R_2$-contractive representation, then $\rho$ is row contractive. If $\rho$ is $C_2$-contractive, then $\rho$ is column contractive.
\end{thm}

\begin{proof}
Let $a_1,\ldots,a_n \in \cal{A}$ with $\sum_{j=1}^n a_ja_j^\ast \leq 1$. We will prove our result by induction on $n$. The case $n=2$ is our hypothesis. Now let $\epsilon >0$ and choose $b\in \cal{A}^{-1}$ such that $\sum_{j=1}^{n-1} a_ja_j^\ast < bb^\ast <1+\epsilon -a_na_n^\ast$. This choice being possible since $\cal{A}$ is logmodular. Set $b_j=b^{-1} a_j$ and note that $\sum_{j=1}^{n-1} b_jb_j^\ast=\sum_{j=1}^{n-1} b^{-1}(\sum_{j=1}^{n-1} a_ja_j^\ast)b^{\ast -1}\leq 1$, hence the induction hypothesis yields $\sum_{j=1}^{n-1}\rho(b_j)^\ast \rho(b_j) \leq 1$ or equivalently 
$$\sum_{j=1}^{n-1} \rho(b^{-1}a_j) \rho(b^{-1}a_j)^\ast =\rho(b)^{-1} (\sum_{j=1}^{n-1} \rho(a_j)\rho(a_j)^\ast) \rho(b)^{\ast -1}\leq 1.$$

Applying the $R_2$ condition again we get
$$\sum_{j=1}^n \rho(a_j)\rho(a_j)^\ast \leq \rho(b)\rho(b)^\ast +\rho(a_n)\rho(a_n)^\ast \leq 1+\epsilon.$$
Since this is true for all $\epsilon$ we have our result.

The case of $C_2$-contractive is similar.
\end{proof}

\begin{prop}\label{unique} Let $\cal{A} \subseteq \cal{B}$ be logmodular, let $\rho: \cal{A} \to B(\cal H)$ be a unital homomorphism and let $h \in \cal H$ be a unit vector. If $\phi, \psi: \cal B \to \bb C$ are states such that $\|\rho(a)h \|^2 \le \phi(a^*a)$ and $\|\rho(a)^*h\|^2 \le \psi(aa^*)$ for all $a \in \cal A^{-1},$ then $\phi = \psi.$ 
\end{prop}

\begin{proof}
For any $a \in \cal A^{-1},$ we have
\begin{align*}
1&=\norm{h}^4=\inp{h}{h}^2=\inp{\rho(a)^{-1}\rho(a) h}{h}^2\\
&= \inp{\rho(a)h}{\rho(a^{-1})^\ast h}^2\leq \norm{\rho(a)h}^2\norm{\rho(a^{-1})^\ast h}^2\\
&\leq  \phi(a^\ast a)\psi(a^{-1}(a^\ast)^{-1})=\phi(a^\ast a)\psi((a^\ast a)^{-1}).
\end{align*}
For $x\in \cal{B}_{sa}$ and $t\in \bb{R}$, $e^{tx}$ is a positive invertible element of $\cal{B}$ and so the above inequality gives
$$1\leq \phi(e^{tx})\psi(e^{-tx}).$$
Let $u(t)=\phi(e^{tx})\psi(e^{-tx})$ and note that this is a differentiable function with a minimum at $t=0$. Taking the derivative and setting $t=0$ we get
\[0=u'(0)=\phi(xe^{tx})-\psi(xe^{-tx})|_{t=0}=\phi(x)-\psi(x).\] Since this holds for all selfadjoint elements in $\cal{B}$ and $\phi,\psi$ are selfadjoint maps we get $\phi=\psi$.
\end{proof}

Note that the above result does not guarantee that such a state $\phi$
exists, nor does it guarantee that if $\phi$ exists satisfying,
$\|\rho(a)h \|^2 \le \phi(a^*a)$, then $\|\rho(a)^*\| \le \phi(aa^*).$

\begin{prop}\label{dominatingrep}
Let $\cal{A}\subseteq \cal{B}$ be logmodular and let $\rho:\cal{A}\to
B(\cal{H}) $ be a representation which is $R_2$-contractive and $C_2$-
contractive. If $h\in \cal{H}$ is fixed, then there exists a
positive linear functional $\phi:\cal{B}\to \bb{C}$ such that
$\phi(1)= \|h\|^2$, $\norm{\rho(a)h}^2\leq \phi(a^\ast a)$, $\norm{\rho(a)^\ast h}^2\leq \phi(aa^\ast)$ and $\phi(a)=\inp{\rho(a)h}{h}$ and such a $\phi$ is unique.
\end{prop}

\begin{proof} It will be enough to consider the case where $\|h\|=1$,
  and prove that there exists a state $\phi$ as above, since uniqueness will follow by Proposition~\ref{unique}.
Since the map $a\mapsto \rho(a)$ is $C_2$-contractive, it follows by Theorem~\ref{tworowimpliescc} that $a\mapsto \rho(a)$ is column contractive. Hence, by Lemma~\ref{colimpliescc} the map $a\mapsto \rho(a)h$ is completely contractive as a map from $\cal{A}$ into the column Hilbert space $\cal{H}_c$. By Theorem~\ref{dominatingstate} there exists a state $\phi_1$ such that 
$\norm{\rho(a)h}^2\leq \phi_1(a^\ast a).$
Moreover, since $a \mapsto \rho(a)$ is $R_2$-contractive, the map $a^\ast \mapsto \rho(a)^\ast h$ is completely contractive as a map from $\cal{A}^\ast$ to the column Hilbert space $\cal{H}_c$. Applying Theorem~\ref{dominatingstate} again yields a state such that 
$\norm{\rho(a)^\ast h}^2\leq \phi_2(aa^\ast)$. From Proposition~\ref{unique} we get that $\phi_1=\phi_2$.

Let $\lambda$ be a complex scalar and consider $\norm{\rho(a+\lambda)h}^2\leq \phi((a+\lambda)^\ast (a+\lambda))$. Expanding both sides and rearranging we get
\[2\Re (\lambda(\inp{\rho(a)h}{h}-\phi(a)))\leq \phi(a^\ast a)-\norm{\rho(a)h}^2.\]
This quantity being positive for all $\lambda\in \bb{C}$ implies that $\inp{\rho(a)h}{h}=\phi(a)$. 
 \end{proof}

\begin{lemma}\label{semispectralstate} Let $\cal A \subseteq \cal B$ be logmodular and let
  $\rho: \cal A \to B(\cal H)$ be a representation that is $R_2$-contractive and
  $C_2$-contractive and for each $h \in \cal H$ let $\phi_h$
  denote the unique positive linear functional as obtained in Proposition~\ref{dominatingrep}, then
  for any vectors, $h,k \in \cal H$, $\phi_{h+k} + \phi_{h-k} = 2 \phi_h
  + 2 \phi_k.$
\end{lemma}
\begin{proof}  The proof is identical to the one outlined in~\cite[Lemma~3]{FS}. For any $a \in \left(\cal A\right)^{-1}$ we have
\begin{align*}
&2(\|h\|^2 + \|k\|^2) \\
=& \|h+k\|^2 + \|h-k\|^2 \\ 
=& \inp{\rho(a)(h+k)}{\rho(a^{-1})^\ast(h+k)}+\inp{\rho(a)(h-k)}{\rho(a^{-1})^\ast (h-k)} \\ 
\le & \left(\|\rho(a)(h+k)\|^2 + \|\rho(a)(h-k)\|^2\right)^{1/2} \\
&\times \left(\|\rho(a^{-1})^*(h+k)\|^2 + \|\rho(a^{-1})^*(h-k)\|^2\right)^{1/2} \\ 
\le & \left( \phi_{h+k}(a^*a) + \phi_{h-k}(a^*a)\right)^{1/2}\left(2\phi_h((a^*a)^{-1}) + 2\phi_k((a^*a)^{-1})\right)^{1/2}. 
\end{align*}
Thus, for any $x=x^* \in \cal B,$ and $t \in \bb R,$ we have that
$$4(\|h\|^2 + \|k\|^2)^2 \le [ \phi_{h+k}(e^{tx}) +
\phi_{h-k}(e^{tx})][2\phi_h(e^{-tx}) + 2\phi_k(e^{-tx})].$$
Since $t=0$ is the minimum of the function on the right hand side,
differentiating and evaluating at $t=0$, yields
\begin{multline*}
0=\left(\phi_{h+k}(x)+ \phi_{h-k}(x)\right)\left(2 \|h\|^2 + 2 \|k\|^2\right) \\- \left(\|h+k\|^2
+ \|h-k\|^2\right)\left(2\phi_h(x)+ 2\phi_k(x)\right),
\end{multline*}
from which the result follows.
\end{proof}

We can now state and prove our key result.

\begin{thm}\label{uniqueextension} Let $\cal A \subseteq \cal B$ be logmodular. If $\rho:
  \cal A \to B(\cal H)$ is a unital homomorphism which is $R_2$-contractive
  and $C_2$-contractive, then there exists a positive map,
  $\Phi: \cal B \to B(\cal H),$ extending $\rho$ and satisfying,
  $\Phi(a^*a) \ge \rho(a)^*\rho(a),$ $\Phi(aa^*) \ge
  \rho(a)\rho(a)^*$, for every $a \in \cal A$. Such a map is unique.
Moreover, if $\Psi: \cal B \to B(\cal H)$ is any completely positive
map extending $\rho,$ then $\Psi = \Phi.$
\end{thm}
\begin{proof}
Let $\phi_h: \cal B \to \bb C, h \in \cal H,$ be the family of
positive functionals guaranteed by Proposition~\ref{dominatingrep}.
If we fix $b \in \cal B,$ then by Lemma~\ref{semispectralstate} the map $h \mapsto \phi_h(b),$ is easily
seen to be a bounded quadratic form on $\cal H$ and hence there exists
a unique bounded operator, $\Phi(b)$ with $\inp{\Phi(b)h}{h}= \phi_h(b).$
The map $b \mapsto \Phi(b)$ is easily checked to be linear, positive and
to satisfy the other conclusions stated in the theorem.

If $\Psi: \cal B \to B(\cal H)$ is any positive map
satisfying the above conclusions, then for each $h \in \cal H,$
the positive linear maps, $\psi_h(b) = \inp{\Psi(b)h}{h}$
satisfy the conclusions of Proposition~\ref{dominatingrep}, and hence
$\psi_h(b)=\phi_h(b),$ for every $b$ and every $h$.

Finally, if $\Psi$ is a completely positive map extending $\rho$, then
by the Cauchy-Schwarz inequality for 2-positive maps, it will
satisfy $\Psi(a^*a) \ge \rho(a)^*\rho(a)$ and $\Psi(aa^*) \ge
\rho(a)\rho(a)^*,$ and so we will have $\Psi =\Phi,$ by the uniqueness
result. 
\end{proof}

Thus, a completely positive extension of $\rho$ exists if and only if
the positive extension that we have constructed is completely
positive. We now re-capture the main result of Foias-Suciu~\cite{FS}.

\begin{cor} Let $X$ be a compact, Hausdorff space, let $\cal A
  \subseteq C(X)$ be logmodular, let $\rho: \cal A \to B(\cal H)$ be a
  unital homomorphism that is  $R_2$-contractive and $C_2$-contractive. Then there exists a Hilbert space $\cal K,$ an isometry
  $V: \cal H \to \cal K,$ and a unital $\ast$-homomorphism $\pi:C(X) \to
  B(\cal K)$ such that $\rho(a) = V^*\pi(a)V,$ for every $a \in \cal
  A.$
If the span of $\pi(C(X))V\cal H$ is dense in $\cal K,$ then this
  representation is unique up to unitary equivalence.
\end{cor} 
\begin{proof} By Theorem~\ref{uniqueextension}, there
  exists a positive map, $\Phi: C(X) \to B(\cal H).$ By Stinespring's
  theorem~\cite{St} (see also~\cite{Pa}) such a map is automatically
  completely positive and hence by Stinespring's dilation
  theorem~\cite{St} (see also~\cite{Pa}) the map $\Phi$ has
  a unique dilation of the above type.

Conversely, given any dilation of the above type, if we set $\Phi(b) =
V^*\pi(b)V,$ then $\Phi$ is a positive map satisfying the above
conditions and hence is unique.
\end{proof}

\section{Matrix Factorization in Logmodular Algebras}

In~\cite[Corollary~18.11]{Pa}, a necessary and sufficient condition is
given for an operator algebra to have the property that every
contractive representation is completely contractive in terms of a
certain type of factorization. In particular, it is proven that a
unital operator algebra $\cal A$ has the property that every
contractive representation is completely contractive if and only if
for every $n$ and for every $(a_{i,j}) \in M_n(\cal A)$ with
$\|(a_{i,j})\| < 1,$ there exists some $m,$ scalar matrices, $C_0,
\ldots, C_m$ (of appropriate sizes) all with norm less than one and
diagonal matrices, $D_1, \ldots, D_m$ (of appropriate sizes), whose
diagonal entries are elements of the open unit ball of $\cal A,$ such
that
\[(a_{i,j}) = C_0D_1C_1 \cdots D_mC_m.\] 
The term {\em appropriate sizes} means simply that the sizes are such
that the product is defined. Combining this result with the result of
Foias-Suciu, yields some new results about uniform logmodular
algebras.

\begin{thm} Let $\cal A \subseteq C(X)$ be a uniform logmodular
  subalgebra.  Then every contractive representation of $\cal A$ is
  completely contractive if and only if for each $f_1, f_2 \in \cal
  A,$ with $|f_1|^2+|f_2|^2 < 1,$ there exists an $m$ and scalar
  matrices, $C_0, \ldots, C_m$ (of appropriate sizes) all with norm
  less than one and diagonal matrices $D_1, \ldots, D_m$ (of
  appropriate sizes) whose diagonal entries are elements of the open
  unit ball of $\cal A,$ such that,
\[\begin{bmatrix} f_1 \\ f_2 \end{bmatrix} = C_0D_1 C_1 \cdots
D_mC_m.\]
\end{thm}
\begin{proof} If $\rho$ is contractive and $D$ is a diagonal matrix in
  $M_n(\cal A)$, then $\norm{\rho_n(D)}\leq \norm{D}$. If the
  factorization condition is met, then we see, by applying $\rho$ to
  both sides of the above equation, that $\rho$ is
  $C_2$-contractive. 

  Note that whenever $\begin{bmatrix} f_1 \\ f_2 \end{bmatrix}$ has a
  factorization as above, then
  \[ \begin{bmatrix} f_1 & f_2 \end{bmatrix} = C_m^tD_m \cdots
  C_1^tD_1C_0^t \]%
  where $C^t$ denotes the transpose.  Since $\|C\| = \|C^t\|$ for
  scalar matrices, it follows from the same argument as in the
  previous paragraph that $\rho$ is also $R_2$-contractive. Hence, by
  the result of Foias-Suciu, the representation $\rho$ is completely
  contractive.

  Conversely, if every contractive representation is completely
  contractive, then by the result cited above, every element of the
  unit ball of $M_n(\cal A)$ has such a factorization. Thus, in
  particular, every element of the unit ball of $M_2(\cal A)$ of the
  form $\begin{bmatrix} f_1 & 0\\ f_2 & 0 \end{bmatrix}$ can be
  factorized, from which the factorization in the statement of the
  theorem follows.
\end{proof}

Of course, the columns of size two in the above theorem can be equally
well replaced by rows of size two.


\begin{thm} Let $\cal A \subseteq C(X)$ be a uniform logmodular subalgebra. 
Then for every $n$ and every $(f_{i,j})$ in the open unit ball of
$M_n(\cal A),$ there exists an $m,$
 scalar matrices, $C_0,
\ldots, C_m$(of appropriate sizes) all with norm less than one and
block diagonal matrices, $D_1, \ldots, D_m$(of appropriate sizes),  whose
direct summands are elements of the open unit ball of either
$M_{2,1}(\cal A)$ or $M_{1,2}(\cal A),$ such
that
\[(f_{i,j}) = C_0D_1C_1 \cdots D_mC_m.\]
\end{thm}
\begin{proof}  Let $\cal C= \cal A$ as algebras but define a new family of
  matrix norms on $\cal C$ by the above factorization.  That is, we
  declare the open unit ball of $M_n(\cal C)$ to be the set of all matrices
  that can be expressed as products of scalar contractions and
  block diagonal matrices as above of norm less than one. Arguing as in~\cite{BP} or in~\cite[Theorem~18.1]{Pa} one shows that this definition yields a
  family of norms on $M_n(\cal C)$ and that these norms satisfy the
  Blecher-Ruan-Sinclair~\cite{BRS} axioms to form an abstract operator algebra. 

The proof will be complete if we can show that the identity map from $\cal A$ to $\cal C$ is a complete isometry.

Clearly, any matrix that is in the unit ball of $M_n(\cal C)$ is also
in the unit ball of $M_n(\cal A).$ This shows that the identity map
from $\cal C$ to $\cal A$ is completely contractive. 

Conversely, since we are allowing contractions from $M_{2,1}(\cal A)$
in our definition of the norm, we have that $\norm{\begin{bmatrix}
    f_1\\f_2 \end{bmatrix}}_{M_{2,1}(\cal C)} \le
\norm{\begin{bmatrix} f_1 \\ f_2 \end{bmatrix}}_{M_{2,1}(\cal A)}.$
Thus, the identity map is also an isometry from $M_{2,1}(\cal A)$ to
$M_{2,1}(\cal C),$ that is, the identity map from $\cal A$ to $\cal C$
is $C_2$-contractive. Similarly, this identity map is $R_2$-contractive.

Since $\cal C$ is an abstract operator algebra, by the
representation theorem of~\cite{BRS} (see~\cite[Corollary~16.7]{Pa} for
an alternate proof), it can be represented completely isometrically on a Hilbert space via some map,
$\rho: \cal C \to B(\cal H).$ Regarding, $\rho: \cal A \to B(\cal H),$
we see that $\rho$ is $C_2$-contractive and $R_2$-contractive on the
uniform logmodular algebra $\cal A,$ so by the Foias-Suciu result, it is a completely contractive map.

Thus, the identity map is completely contractive as a map from $\cal
A$ to $\cal C.$ As we saw earlier, the identity map is completely
contractive from $\cal C$ to $\cal A,$ and hence the identity map is a
complete isometry. Thus, the original family of matrix norms on $\cal A$ is equal to the new family. 
\end{proof}


The above result can also be used as the basis of an approach to
deciding whether or not the analogue of the Foias-Suciu result is true
for logmodular subalgebras of non-commutative $C^\ast$-algebras. Given a
logmodular subalgebra $\cal A \subseteq \cal B,$ one can form a new
operator algebra $\cal C$ by endowing $\cal A$ with a (possibly) new
operator algebra norm as above, allowing the diagonals to be direct
sums of block diagonal matrices where each block is either a $2 \times
1$ matrix or a $1 \times 2$ matrix over $\cal A.$ The proof of the
above result shows that every $\rho: \cal A \to B(\cal H)$ that is
both $R_2$- and $C_2$-contractive will be completely
contractive if and only if the identity map from $\cal A$ to $\cal C$
is a complete isometry.  Thus, some progress could conceivably be made
by studying the abstract operator algebra $\cal C.$ These
considerations motivate the following problem.

\begin{ques} Let $\cal A \subseteq \cal B$ be a logmodular subalgebra
  of a $C^\ast$-algebra and let $\cal C$ be the abstract operator algebra
  defined in the above paragraph. Is $\cal C$ a logmodular algebra?
\end{ques}

\end{document}